\documentclass[a4,12pt]{amsart}
\usepackage{amssymb,amstext,amsmath,amscd,amsthm,amsfonts,enumerate,latexsym}
\usepackage{color}
\usepackage[dvipdfmx]{graphicx}
\usepackage[all]{xy}
\usepackage{stmaryrd} 

\usepackage{bm}

\usepackage[dvipdfmx]{hyperref}

\usepackage{mathptmx}
\theoremstyle{plain}
\newtheorem{thm}{Theorem}[section]

\newtheorem*{thm*}{Theorem}
\newtheorem*{cor*}{Corollary}

\newtheorem{lem}[thm]{Lemma}
\newtheorem{cor}[thm]{Corollary}

\newtheorem*{claim*}{Claim}

\theoremstyle{definition}

\newtheorem{ex}[thm]{Example}

\theoremstyle{remark}
\newtheorem{rem}[thm]{Remark}

\numberwithin{equation}{thm}

\def\Im{\operatorname{Im}}

\def\Ker{\operatorname{Ker}}

\def\bbZ{\mathbb{Z}}
\def\bbN{\mathbb{N}}

\def\m{\mathfrak m}

\newcommand{\rma}{\mathrm{a}}

\newcommand{\rmc}{\mathrm{c}}

\newcommand{\rme}{\mathrm{e}}
\newcommand{\rmf}{\mathrm{f}}

\newcommand{\rmr}{\mathrm{r}}

\newcommand{\rmI}{\mathrm{I}}

\newcommand{\rmK}{\mathrm{K}}

\newcommand{\rmQ}{\mathrm{Q}}

\newcommand{\fkm}{\mathfrak{m}}

\newcommand{\mapright}[1]{%
\smash{\mathop{%
\hbox to 1cm{\rightarrowfill}}\limits^{#1}}}

\newcommand{\mapleft}[1]{%
\smash{\mathop{%
\hbox to 1cm{\leftarrowfill}}\limits_{#1}}}

\title[Remarks on almost Gorenstein rings]{Remarks on almost Gorenstein rings}

\author[Naoki Endo]{Naoki Endo}
\address{School of Political Science and Economics, Meiji University, 1-9-1 Eifuku, Suginami-ku, Tokyo 168-8555, Japan}
\email{endo@meiji.ac.jp}
\urladdr{https://www.isc.meiji.ac.jp/~endo/}

\author[Naoyuki Matsuoka]{Naoyuki Matsuoka}
\address{Department of Mathematics, School of Science and Technology, Meiji University, 1-1-1 Higashi-mita, Tama-ku, Kawasaki 214-8571, Japan}
\email{naomatsu@meiji.ac.jp}

\thanks{2020 {\em Mathematics Subject Classification.} 13H10, 13A15, 13A02.}
\thanks{{\em Key words and phrases.} Almost Gorenstein local ring, Almost Gorenstein graded ring, numerical semigroup}
\thanks{The first author was partially supported by JSPS Grant-in-Aid for Young Scientists 20K14299 and JSPS Grant-in-Aid for Scientific Research (C) 23K03058.}

\begin{document}

\maketitle

\setlength{\baselineskip} {15pt}

\begin{abstract}
This paper investigates the relation between the almost Gorenstein properties for graded rings and for local rings. Once $R$ is an almost Gorenstein graded ring, the localization $R_M$ of $R$ at the graded maximal ideal $M$ is almost Gorenstein as a local ring. The converse does not hold true in general (\cite[Theorems 2.7, 2.8]{GMTY2}, \cite[Example 8.8]{GTT}). However, it does for one-dimensional graded domains with mild conditions, which we clarify in the present paper. We explore the defining ideals of almost Gorenstein numerical semigroup rings as well.
\end{abstract}



\section{Introduction}



{\it An almost Gorenstein ring}, which we focus on in the present paper, is one of the attempts to generalize Gorenstein rings. The motivation for this generalization comes from the strong desire to stratify Cohen-Macaulay rings, finding new and interesting classes which naturally include that of Gorenstein rings. The theory of almost Gorenstein rings was introduced by Barucci and Fr{\"o}berg \cite{BF} in the case where the local rings are analytically unramified and of dimension one, e.g., numerical semigroup rings over a field. In 2013, their work inspired Goto, the second author of this paper, and Phuong \cite{GMP} to extend the notion of almost Gorenstein rings for arbitrary one-dimensional Cohen-Macaulay local rings. 
More precisely, a one-dimensional Cohen-Macaulay local ring $R$ is called {\it almost Gorenstein} if $R$ admits a canonical ideal $I$ of $R$ such that $\rme_1(I) \le \rmr(R)$, where $\rme_1(I)$ denotes the first Hilbert coefficients of $R$ with respect to $I$ and $\rmr(R)$ is the Cohen-Macaulay type of $R$ (\cite[Definition 3.1]{GMP}). Two years later, Goto, Takahashi, and the first author of this paper \cite{GTT} defined the almost Gorenstein graded/local rings of arbitrary dimension. 
Let $(R, \m)$ be a Cohen-Macaulay local ring. Then $R$ is said to be an {\it almost Gorenstein local ring} if $R$ admits a canonical module $\rmK_R$ and there exists an exact sequence
$$
0 \to R \to \rmK_R \to C \to 0
$$ 
of $R$-modules such that $\mu_R(C) = \rme^0_\m(C)$ (\cite[Definition 3.3]{GTT}). Here, $\mu_R(-)$ (resp. $\rme^0_\m(-)$) denotes the number of elements in a minimal system of generators (resp. the multiplicity with respect to $\m$). When $\dim R=1$, if $R$ is an almost Gorenstein local ring in the sense of \cite{GTT}, then $R$ is almost Gorenstein in the sense of \cite{GMP}, and vice versa provided the field $R/\fkm$ is infinite (\cite[Proposition 3.4]{GTT}). 
However, the converse does not hold in general (\cite[Remark 3.5]{GTT}, see also \cite[Remark 2.10]{GMP}). 
Similarly as in local rings, a Cohen-Macaulay graded ring $R=\oplus_{n \ge 0}R_n$ with $k=R_0$ a local ring is called an {\it almost Gorenstein graded ring} if $R$ admits a graded canonical module $\rmK_R$ and there exists an exact sequence
$$
0 \to R \to \rmK_R(-a) \to C \to 0
$$ 
of graded $R$-modules with $\mu_R(C) = \rme^0_M(C)$ (\cite[Definition 8.1]{GTT}). Here, $M$ is the graded maximal ideal of $R$, $a=\rma(R)$ is the a-invariant of $R$, and $\rmK_R(-a)$ denotes the graded $R$-module whose underlying $R$-module is the same as that of $\rmK_R$ and whose grading is given by $[\rmK_R(-a)]_n = [\rmK_R]_{n-a}$ for all $n \in \Bbb Z$.

Every Gorenstein local/graded ring is almost Gorenstein. 
The definitions assert that once $R$ is an almost Gorenstein local (resp. graded) ring, either $R$ is a Gorenstein ring, or even though $R$ is not a Gorenstein ring, the local (resp. graded) ring $R$ is embedded into the module $\rmK_R$ (resp. the graded module $\rmK_R(-a)$) and the difference $C$ behaves well.
Moreover, if $R$ is an almost Gorenstein graded ring, then the localization $R_M$ of $R$ at $M$ is an almost Gorenstein local ring, which readily follows from the definition. However, 
the converse does not hold true in general (\cite[Theorems 2.7, 2.8]{GMTY2}, \cite[Example 8.8]{GTT}), even though it does for determinantal rings of generic, as well as symmetric, matrices over a field (\cite[Theorem 1.1]{CELW}, \cite[Theorem 1.1]{T}). 

In this paper we investigate the question of when the converse holds in one-dimensional rings. Throughout this paper, unless otherwise specified, let $R = \bigoplus_{n \ge 0}R_n$ be a one-dimensional Noetherian $\bbZ$-graded integral domain admitting a graded canonical module $\rmK_R$. We assume $k=R_0$ is a field, and $R_n \ne (0)$ and $R_{n+1} \ne (0)$ for some $n \ge 0$. 
Let $M$ denote the graded maximal ideal of $R$.

With this notation this paper aims at proving the following result.

\begin{thm}\label{main}
There exists a graded canonical ideal $J$ of $R$ containing a parameter ideal as a reduction,  
and the following conditions are equivalent. 
\begin{enumerate}[$(1)$]
\item $R$ is an almost Gorenstein graded ring. 
\item $R_{M}$ is an almost Gorenstein local ring in the sense of \cite[Definition 3.1]{GMP}.
\item $R_{M}$ is an almost Gorenstein local ring in the sense of \cite[Definition 3.3]{GTT}.
\end{enumerate}
\end{thm}

Theorem \ref{main} guarantees the existence of a (graded) canonical ideal admitting a parameter ideal as a reduction; hence the proof of \cite[Proposition 3.4]{GTT} shows that the conditions $(2)$ and $(3)$ are equivalent even though the field $R/M$ is finite. As we mentioned, by the definition of almost Gorenstein local/graded rings we only need to verify the implication $(2) \Rightarrow (1)$. What makes $(2)\Rightarrow (1)$ interesting and difficult is that the implication is not true in general.




Let us now explain how this paper is organized. We prove Theorem \ref{main} in Section 2 after preparing some auxiliaries. We also explore the explicit generators of defining ideals of almost Gorenstein numerical semigroup rings. Section 3 is devoted to providing examples illustrating Theorem \ref{main}.


\section{Proof of Theorem \ref{main}}

Let $S$ be the set of non-zero homogeneous elements in $R$. The localization $S^{-1}R = K[t, t^{-1}]$ of $R$ with respect to $S$ is a {\it simple} graded ring, i.e., every non-zero homogeneous element is invertible, where $t$ is a homogeneous element of degree $1$ (remember that $R_n \ne (0)$ and $R_{n+1} \ne (0)$) which is transcendental over $k$, and $K=[\, S^{-1}R\, ]_0$ is a field. 
Let $\overline{R}$ be the integral closure of $R$ in its quotient field $\rmQ(R)$. 

We begin with the following, which has already appeared in \cite[Lemma 2.1]{E}. Because it plays an important role in our argument, we include a brief proof for the sake of completeness. 

\begin{lem}
The equality $\overline{R} = K[t]$ holds in $\rmQ(R)$. 
\end{lem}

\begin{proof}
As $R$ is an integral domain, we obtain that $\overline{R}$ is a graded ring and $\overline{R} \subseteq S^{-1}R=K[t, t^{-1}]$ (\cite[page 157]{ZS}). Since the field $k$ is Nagata, so is the finitely generated $k$-algebra $R$. Hence $\overline{R}$ is a module-finite extension of $R$. One can verify that $R_n = (0)$ for all $n<0$ and $R_0 = k$; hence $[\,\overline{R}\,]_n = (0)$ for all $n<0$, $L=[\,\overline{R}\,]_0$ is a field, and $k \subseteq L \subseteq K$. We set $N = \bigoplus_{n > 0}[\,\overline{R}\,]_n$. Since the local ring $\overline{R}_N$ of $\overline{R}$ at the maximal ideal $N$ is a DVR, the ideal $N$ is principal. Choose a homogeneous element $f \in \overline{R}$ of degree $q>0$ with $N = f\overline{R}$. Then 
$$
\overline{R} = L[N] = L[f] \subseteq S^{-1}R =K[t, t^{-1}]. 
$$ 
Besides, because $\overline{R}[f^{-1}] = L[f, f^{-1}]$ is a simple graded ring and $R \subseteq \overline{R}[f^{-1}]$, we have $S^{-1}R \subseteq \overline{R}[f^{-1}]=L[f, f^{-1}]$. Thus $K[t, t^{-1}] = L[f, f^{-1}]$, so that $K=L$ and $q = 1$. Consequently, $\overline{R} = L[f] = K[f] = K[t]$, as desired. 
\end{proof}

We are ready to prove Theorem \ref{main}.

\begin{proof}[Proof of Theorem \ref{main}]
Since $S^{-1}\rmK_R \cong S^{-1}R$ as a graded $S^{-1}R$-modules, we have an injective homomorphism $0 \to \rmK_R \overset{\varphi}{\to} S^{-1}R$ of graded $R$-modules. Choose $s \in S$ such that $s\cdot \varphi(\rmK_R) \subsetneq R$. Set $J = s\cdot \varphi(\rmK_R)$ and $q = -\deg s$. Then $\rmK_R(q) \cong J$ as a graded $R$-module.
By setting $a=\rma(R)$ and $\ell = -(q+a)$, we then have $J_{\ell} \ne (0)$ and $J_n =(0)$ for all $n< \ell$. We now choose a non-zero homogeneous element $f \in J_{\ell}$. Note that $\ell >0$ and $f \in \left[\ \!  \overline{R}\ \! \right]_{\ell}$. Therefore 
$$
J \overline{R} = t^{\ell}\overline{R} = f\overline{R}. 
$$
This shows the equality $J^{r+1} = f J^r$ holds where $r = \mu_R(\overline{R})-1$. Here, we recall that $\mu_R(-)$ denotes the number of elements in a minimal system of generators. Hence, $J$ is a graded canonical ideal of $R$ which contains a parameter ideal $(f)$ of $R$ as a minimal reduction. 

As for the equivalent conditions, we only need to show the implication $(2) \Rightarrow (1)$. 

$(2) \Rightarrow (1)$ 
We consider the exact sequence
$$
0 \to R \overset{\psi}{\to} J(\ell) \to C \to 0
$$
of graded $R$-modules with $\psi(1) = f$. Since $\m I = \m f$, we get $MJ=Mf$ and hence $MC = (0)$, i.e. $C$ is an Ulrich $R$-module. Therefore $R$ is an almost Gorenstein graded ring because $J(\ell) \cong \rmK_R(-a)$. This completes the proof.
\end{proof}


Let $\bbN$ be the set of non-negative integers. A {\it numerical semigroup} is a non-empty subset $H$ of $\bbN$ which is closed under addition, contains the zero element, and whose complement in $\bbN$ is finite. Every numerical semigroup $H$ admits a finite minimal system of generators, i.e., there exist positive integers $a_1, a_2, \ldots, a_{\ell} \in H~(\ell \ge 1)$ with $\gcd(a_1, a_2,\ldots , a_{\ell}) =1$ such that 
$H = \left<a_1, a_2, \ldots, a_\ell\right>=\left\{\sum_{i=1}^\ell c_ia_i ~\middle|~  c_i \in \Bbb N~\text{for~all}~1 \le i \le \ell \right\}$. For a field $k$, the ring $k[H] = k[t^{a_1}, t^{a_2}, \ldots, t^{a_\ell}]$ (or $k[[H]] =k[[t^{a_1}, t^{a_2}, \ldots, t^{a_\ell}]]$) is called the {\it numerical semigroup ring} of $H$ over $k$, where $t$ denotes an indeterminate over $k$. Note that the ring $k[H]$ satisfies the assumption of Theorem \ref{main}.  Let $M$ be the graded maximal ideal of $R$. 
Since every non-zero ideal in numerical semigroup rings admits its minimal reduction, the two definitions for almost Gorenstein local rings \cite[Definition 3.1]{GMP} and \cite[Definition 3.3]{GTT} are equivalent.
In addition, the local ring $k[H]_M$ is almost Gorenstein if and only if $k[[H]]$ is an almost Gorenstein local ring; equivalently, the semigroup $H$ is almost symmetric (\cite[Proposition 29]{BF}, \cite[Theorem 2.4]{N}). 

Hence we have the following, which recovers a result of Goto, Kien, Matsuoka, and Truong.

\begin{cor}[{\cite[Proposition 2.3]{GKMT}}]
A numerical semigroup ring $k[[H]]$ is an almost Gorenstein local ring if and only if $k[H]$ is an almost Gorenstein graded ring, or equivalently, $H$ is almost symmetric. 
\end{cor}

In the rest of this section, let $R=k[H]$ be the numerical semigroup ring over $k$ and 
$$
\rmc(H) = \min \{n \in \Bbb Z \mid m \in H~\text{for~all}~m \in \Bbb Z~\ \! \text{with~}\ \! m \ge n\}
$$
the conductor of $H$. We set $\rmf(H) = \rmc(H) -1$. Then $\rmf(H) = \max ~(\Bbb Z \setminus H)$, which is called the Frobenius number of $H$. 
Let 
$$
\mathrm{PF}(H) = \{n \in \Bbb Z \setminus H \mid n + a_i \in H~\text{for~all}~1 \le  i \le \ell\}
$$ 
be the set of pseudo-Frobenius numbers of $H$. The graded canonical module $\rmK_R$ has the form 
$$
\rmK_R = \sum_{c \in \mathrm{PF}(H)}Rt^{-c}
$$
whence $\rmf(H) = \rma(R)$ and $\# \mathrm{PF}(H) = \rmr(R)$ (\cite[Example (2.1.9), Definition (3.1.4)]{GW}). Here, $\rmr(R)$ denotes the Cohen-Macaulay type of $R$.  
We write $\mathrm{PF}(H) =\{c_1 < c_2 < \cdots < c_r\}$; hence $r=\rmr(R)$ and $c_r = \rma(R)$. For each $1 \le i \le r$, we set $b_i = -c_{r+1-i}$. Thus $\rmK_R = \sum_{i=1}^r Rt^{b_i}$. 

Let $S = k[x_1, x_2, . . . , x_{\ell}]$ be the weighted polynomial ring over the field $k$ with $x_i \in S_{a_i}$ for all $1 \le i \le \ell$. Consider the homomorphism $\varphi : S \to R$ of graded $k$-algebras defined by $\varphi (x_i) = t^{a_i}$ for all $1 \le i \le \ell$.

Suppose $R$ is almost Gorenstein, but not a Gorenstein ring. Then $r \ge 2$. Since 
$R \subseteq t^{-b_1}\rmK_R = \sum_{i=1}^rRt^{b_i - b_1} \subseteq \overline{R}$, we have $M\rmK_R \subseteq Rt^{b_1}$, where $M=(t^h \mid 0<h \in H)$ is the graded maximal ideal of $R$. Hence, $c_r - c_i = c_{r-i}$ for all $1 \le i \le r-1$ 
 (see e.g., \cite[Proposition 2.3]{GKMT}, \cite[Theorem 3.11]{GMP}, \cite[Definition 8.1]{GTT}, and \cite[Theorem 2.4]{N}). We consider the graded $S$-linear map 
 \begin{equation*}
F=
\begin{matrix}
S\left(-b_1\right) \\
\oplus \\
S\left(-b_2\right) \\
\oplus \\
\vdots \\
\oplus \\
S\left(-b_r\right)
\end{matrix} \overset{\varepsilon}{\longrightarrow} \rmK_R \longrightarrow 0
\end{equation*}
defined by $\varepsilon(\bm{e}_i) = t^{b_i}$ for each $2 \le i \le \ell$ and $\varepsilon(\bm{e}_1) = -t^{b_1}$, where $\{\bm{e}_i\}_{1 \le i\le r}$ denotes the standard basis of $F$. Set $L = \Ker \varepsilon$. For each $2 \le i \le r$ and $1 \le j \le \ell$, we have $x_j t^{b_i} \in Rt^{b_1}$. Choose a homogeneous element $y_{ij} \in S$ of degree $a_j + b_i - b_1$ such that $x_j t^{b_i} - y_{ij} t^{b_1} = 0$; hence $x_j \bm{e}_i + y_{ij}\bm{e}_1 \in L$. Since $\{x_j \bm{e}_i + y_{ij}\bm{e}_1\}_{2 \le i \le r, \ 1 \le j \le \ell}$ forms a part of minimal basis of $L$, we have $q \ge (r-1)\ell$, where $q = \mu_S(L)$. Let $m = q-(r-1)\ell$. If $m > 0$, then we can choose homogeneous elements $z_1, z_2, \ldots, z_m \in S$ such that 
$\{x_j \bm{e}_i + y_{ij}\bm{e}_1, z_1\bm{e}_1, z_2\bm{e}_1, \ldots, z_m\bm{e}_1\}_{2 \le i \le r, \ 1 \le j \le \ell}$ is a minimal basis of $L$. Hence 
\begin{equation*}
G \overset{\Bbb M}{\longrightarrow}
F\overset{\varepsilon}{\longrightarrow} \rmK_R \longrightarrow 0
\end{equation*}
gives a minimal presentation of $\rmK_R$, where $G$ denotes a graded free $S$-module of rank $q$ and the $r \times q$ matrix $\Bbb M$ has the following form
$$
\Bbb M = 
\begin{bmatrix}
y_{21} ~y_{22}~ \cdots ~y_{2\ell} & y_{31}~ y_{32}~ \cdots ~y_{3\ell} & \cdots & y_{r1}~ y_{r2}~ \cdots ~y_{r\ell} & z_1 ~z_2 ~\cdots ~z_m \\
\ \  x_1 \ \  x_2\ \  \cdots \ \  x_{\ell} \ \  & 0 & 0 & 0  & 0 \\
0 & \ \ x_1 \ \ x_2 \ \ \cdots \ \ x_{\ell}\ \  & 0 & 0  & 0 \\
\vdots & \vdots & \ddots & \vdots & \vdots \\
0 & 0 & 0 & \ \ x_1 \ \ x_2 \ \ \cdots\ \  x_{\ell} \ \ & 0
\end{bmatrix}.
$$
Note that $b_i - b_1 = c_{i-1}$ and $a_j + (b_i - b_1) \in H$ (remember that $c_r - c_i = c_{r-i}$). We write $a_j + (b_i - b_1) = d_1 a_1 + d_2a_2 + \cdots + d_{\ell}a_{\ell}$ for some $d_i \in \bbN$. Then, because $b_i - b_1 \not\in H$, we have $d_j =0$. As $y_{ij}$ has degree $a_j + b_i -b_1$, we may choose 
$$
y_{ij} = \prod_{1 \le k \le \ell, \ k \ne j}x_k^{d_k} \ \ \ \text{for all}\ \ \ 2 \le i \le r, \ 1 \le j \le \ell.
$$ 
With this notation we reach the following, where, for each $t \ge 1$, $\rmI_t(\Bbb X)$ denotes the ideal of $S$ generated by the $t \times t$ minors of a matrix $\Bbb X$.


\begin{thm}\label{2.3}
Suppose that $R=k[H]$ is almost Gorenstein, but not a Gorenstein ring. Then, for each $2 \le i \le r$, the difference $\deg y_{ij} - \deg x_j \ (= c_{i-1})$ is constant for every $1 \le j \le \ell$, and the defining ideal of $R$ has the following form
$$
\Ker \varphi = \sum_{i=2}^r
{\rmI_2}\begin{pmatrix}
y_{i1} & y_{i2} & \cdots & y_{i\ell} \\
x_1 & x_2 & \cdots & x_{\ell}
\end{pmatrix}
+ (z_1, z_2, \cdots, z_m).
$$
\end{thm}

\begin{rem}
For a higher dimensional {\it semi-Gorenstein} ring $A$, i.e., a special class of almost Gorenstein ring, the form of the defining ideals can be determined by the minimal free resolution of $A$ (\cite[Theorem 7.8]{GTT}). Our contribution in Theorem \ref{2.3} is that we succeeded in writing $y_{ij}$ concretely in case of numerical semigroup rings.
\end{rem}

\begin{rem} 
When the almost symmetric semigroup $H$ is minimally generated by four elements, 
Eto provided an explicit minimal system of generators of defining ideals of the semigroup rings $k[H]$ by using the notion of RF-matrices (\cite[Section 5]{Eto}). 
\end{rem}


\begin{ex}
The semigroup ring $R=k[H]$ for a numerical semigroup $H$ described below is an almost Gorenstein graded ring, and its defining ideal $\Ker \varphi$ is given by the following form. 
\begin{enumerate}[$(1)$]
\item Let $H=\left<7,8,13,17,19\right>$. Then $\mathrm{PF}(H) = \{6,9,12,18\}$ and
\begin{eqnarray*}
\Ker \varphi \!\!&=&\!\! {\rmI_2}
\begin{pmatrix} 
x_3 & x_1^2 & x_5 & x_1x_2^2 & x_2x_4\\
x_1 & x_2 & x_3 & x_4 & x_5 
\end{pmatrix}
 \ + \ {\rmI_2} 
 \begin{pmatrix} 
 x_2^2 & x_4 & x_1^2x_2 & x_3^2 & x_1x_2x_3\\
x_1 & x_2 & x_3 & x_4 & x_5 
 \end{pmatrix} \\
 \!\!&+&\!\! {\rmI_2} 
 \begin{pmatrix} 
 x_5 & x_1x_3 & x_2x_4 & x_1^3x_2 & x_1^2x_4\\
x_1 & x_2 & x_3 & x_4 & x_5 
 \end{pmatrix}.
\end{eqnarray*}
\item Let $H=\left<11,13,14,16,31\right>$. Then $\mathrm{PF}(H) = \{15,17,19,34\}$ and
\begin{eqnarray*}
\Ker \varphi \!\!&=&\!\! {\rmI_2} 
\begin{pmatrix} 
x_2^2 & x_3^2 & x_2x_4 & x_5 & x_1^3x_2\\
x_1 & x_2 & x_3 & x_4 & x_5 
\end{pmatrix} 
 \ + \ {\rmI_2} 
\begin{pmatrix} 
x_3^2 & x_3x_4 & x_5 & x_1^3 & x_4^3\\
x_1 & x_2 & x_3 & x_4 & x_5 
\end{pmatrix} \\
 \!\!&+&\!\! {\rmI_2} 
\begin{pmatrix} 
x_3x_4 & x_4^2 & x_1^3 & x_1^2x_2 & x_1x_2^3\\
x_1 & x_2 & x_3 & x_4 & x_5 
\end{pmatrix} 
 \ + \ (x_1x_4-x_2x_3).
\end{eqnarray*}
\item Let $H=\left<13,15,16,18,19\right>$. Then $\mathrm{PF}(H) = \{17,20,23,40\}$ and
\begin{eqnarray*}
\Ker \varphi \!\!&=&\!\! {\rmI_2} 
\begin{pmatrix} 
x_2^2 & x_3^2 & x_2x_4 & x_3x_5 & x_4^2\\
x_1 & x_2 & x_3 & x_4 & x_5 
\end{pmatrix}
 \ + \ {\rmI_2}
\begin{pmatrix} 
x_2x_4 & x_3x_5 & x_4^2 & x_5^2 & x_1^3\\
x_1 & x_2 & x_3 & x_4 & x_5
\end{pmatrix} \\
 \!\!&+&\!\!  {\rmI_2} 
\begin{pmatrix} 
x_4^2 & x_5^2 & x_1^3 & x_1^2x_2 & x_1^2x_3\\
x_1 & x_2 & x_3 & x_4 & x_5
 \end{pmatrix}
\ + \ (x_1x_4-x_2x_3, x_1x_5-x_3^2, x_2x_5-x_3x_4).
\end{eqnarray*}
\end{enumerate}
\end{ex}


The next provides an explicit minimal system of generators for defining ideals of $R=k[H]$ when $R$ has {\it minimal multiplicity}, i.e., the embedding dimension is equal to the multiplicity. 

\begin{cor}[cf. {\cite[Corollary 7.10]{GTT}}]
Suppose that $R=k[H]$ is almost Gorenstein, but not a Gorenstein ring. 
If $R$ has minimal multiplicity, 
the defining ideal of $R$ has the following form
$$
\Ker \varphi = \sum_{i=2}^r
{\rmI_2}\begin{pmatrix}
y_{i1} & y_{i2} & \cdots & y_{i\ell} \\
x_1 & x_2 & \cdots & x_{\ell}
\end{pmatrix}.
$$
\end{cor}

\begin{proof}
We maintain the notation as in this section. Since $R$ has minimal multiplicity, by \cite[Theorem 1]{Sally} $q = (\ell-2) \binom{\ell}{\ell-1} = (\ell-2)\ell=(r-1)\ell$, so that $m=0$.
\end{proof}

\section{Examples of Theorem \ref{main}}

We close this paper by providing some examples. 
In this section, the almost Gorenstein property for local rings refers to the definition in the sense of \cite[Definition 3.3]{GTT}.
The first example indicates that Theorem \ref{main} does not hold unless $R$ is an integral domain. 

\begin{ex}[{\cite[Example 8.8]{GTT}}]\label{3.1}
Let $U=k[s, t]$ be the polynomial ring over a field $k$ and set $R = k[s, s^3t, s^3t^2, s^3t^3]$. We regard $U$ as a $\bbZ$-graded ring under the grading $k=U_0$ and $s, t \in U_1$. Let $M$ be the graded maximal ideal of $R$. Then the following assertions hold true. 
\begin{enumerate}[$(1)$]
\item $R$, $R/sR$ are not almost Gorenstein graded rings. 
\item $R_{M}$, $R_{M}/sR_{M}$ are almost Gorenstein local rings.
\end{enumerate}
\end{ex}

\begin{proof}
Let $S=k[X, Y, Z, W]$ be the polynomial ring over $k$. We consider $S$ as a $\bbZ$-graded ring with $k=S_0$, $X \in S_1$, $Y \in S_4$, $Z \in S_5$, and $W \in S_6$. Let $\varphi : S \to R$ be the graded $k$-algebra map such that $\varphi(X) = s$, $\varphi(Y) = s^3t$, $\varphi(Z) = s^3t^2$, and $\varphi(W) = s^3t^3$. By \cite[Example 8.8]{GTT}, $R$ is not almost Gorenstein graded ring with $\rma(R) = -2$, but the local ring $R_{M}$ is almost Gorenstein.

The exact sequence $0 \to R(-1) \overset{s}{\to} R \to R/sR \to 0$ of graded $R$-modules induces the isomorphism $\rmK_{(R/sR)} \cong (\rmK_R/s\rmK_R)(1)$. 
Note that $\rma(R/sR) = -1$. 
If $R/sR$ is an almost Gorenstein graded ring, we can choose an exact sequence 
$$
0 \to R/sR \overset{\Psi}{\to} (\rmK_R/s\rmK_R)(2) \to D \to 0
$$
of graded $R$-modules such that $MD=(0)$. Write $\Psi(1) = \overline{\xi}$ with $\xi \in [\rmK_R]_2$. We consider
$$
R \overset{\Phi}{\to} \rmK_R(2) \to C \to 0
$$
where $\Phi(1)=\xi$. Then $C/sC \cong D$. As $MD=(0)$, we get $\dim_RC \le 1$, Hence the map $\Phi$ is injective (\cite[Lemma 3.1]{GTT}), and $s$ is a non-zerodivisor on $C$ because $R/sR \otimes_R \Phi = \Psi$. Thus $\mu_R(C) = \rme^0_M(C)$. This makes a contradiction. As $X$ is superficial for $S/(Y, Z, W)$ with respect to the maximal ideal of $S$, by \cite[Theorem 3.7 (2)]{GTT}, we conclude that $R_M/sR_M$ is almost Gorenstein as a local ring. 
\end{proof}

\begin{rem}
We maintain the same notation as in Example \ref{3.1}. Let $T=k[Y, Z, W]$ be the polynomial ring over $k$. 
Note that $R/sR \cong T/(YW-Z^2, YZ, Y^2) = T/
\rmI_2
\left(
\begin{smallmatrix}
0 & Y & Z \\
Y & Z & W 
\end{smallmatrix}
\right) =V$. 
If we consider $T$ as a $\bbZ$-graded ring under the grading $k=T_0$, $Y \in T_4$, $Z \in T_5$, and $W \in T_6$, as we have shown in Example \ref{3.1} the ring $V$ is not almost Gorenstein as a graded ring. Whereas, if we consider $T$ as a $\bbZ$-graded ring with 
$k=T_0$ and $Y, Z, W \in T_1$, the $T$-module $V$ has a graded minimal free resolution of the form
\begin{equation*}
0 \longrightarrow
\begin{matrix}
T\left(-3\right) \\
\oplus \\
T\left(-3\right) 
\end{matrix}
\overset{\ 
\left[\begin{smallmatrix}
0 & Y \\
Y & Z \\
Z & W
\end{smallmatrix}\right] \ 
}{\longrightarrow}
\begin{matrix}
T\left(-2\right) \\
\oplus \\
T\left(-2\right) \\
\oplus \\
T\left(-2\right) 
\end{matrix}
\overset{[\Delta_1 \ -\Delta_2 \ \Delta_3]}{\longrightarrow}
T \overset{\varepsilon}{\longrightarrow} V \longrightarrow 0
\end{equation*}
where $\Delta_1 = YW-Z^2$, $\Delta_2 = -YZ$, and $\Delta_3 = -Y^2$.
Taking $\rmK_T$-dual, we get the resolution 
 \begin{equation*}
0 \longrightarrow
T\left(-3\right) \\
\overset{\ 
\left[\begin{smallmatrix}
\ \  \Delta_1 \\
-\Delta_2 \\
\ \ \Delta_3
\end{smallmatrix}\ \right] \ 
}{\longrightarrow}
\begin{matrix}
T\left(-1\right) \\
\oplus \\
T\left(-1\right) \\
\oplus \\
T\left(-1\right) 
\end{matrix}
\overset{\ 
\left[\begin{smallmatrix}
0 & Y & Z\\
Y & Z & W \\
\end{smallmatrix}\right]
}{\longrightarrow}
\begin{matrix}
T \\
\oplus \\
T
\end{matrix} \overset{\varepsilon}{\longrightarrow} \rmK_V \longrightarrow 0
\end{equation*}
of $\rmK_V$ as a graded $T$-module. We then consider the homomorphism
$$
V \overset{\Phi}{\to} \rmK_V \to C \to 0
$$
of graded $T$-modules defined by $\Phi(1) = \xi$, where $\xi = \varepsilon \left(
\left(\begin{smallmatrix}
1 \\
0
\end{smallmatrix}\right)\right)$. The isomorphisms $C \cong \rmK_V/V\xi \cong T/(Y, Z, W)$ guarantee that $\Phi$ is injective and $NC =(0)$ where $N=(Y, Z, W)T$. Thus $V$ is an almost Gorenstein graded ring. Hence the almost Gorenstein property for graded rings depends on the choice of its gradings. 
\end{rem}

\begin{ex}\label{3.3}
Let $S=k[X, Y, Z]$ be the polynomial ring over a field $k$. We consider $S$ as a $\bbZ$-graded ring under the grading $k=S_0$, $X \in S_3$, $Y \in S_1$, and $Z \in S_2$. Set $R = S/(Z^3- X^2, XY, YZ)$. Then $R$ is not an almost Gorenstein graded ring, but the local ring $R_M$ is almost Gorenstein, where $M$ denotes the graded maximal ideal of $R$.
\end{ex}

\begin{proof}
Let $I=(Z^3- X^2, XY, YZ)$. Note that $I=(X, Z) \cap (Z^3-X^2, Y) = 
\rmI_2
\left(
\begin{smallmatrix}
Z^2 & X & Y \\
X & Z & 0
\end{smallmatrix}
\right)$. Thus $R$ is a Cohen-Macaulay reduced ring with $\dim R=1$. 
Note that
\begin{equation*}
0 \longrightarrow
\begin{matrix}
S\left(-7\right) \\
\oplus \\
S\left(-6\right) 
\end{matrix}
\overset{\ 
\left[\begin{smallmatrix}
Z^2 & X \\
X & Z \\
Y & 0
\end{smallmatrix}\right] \ 
}{\longrightarrow}
\begin{matrix}
S\left(-3\right) \\
\oplus \\
S\left(-4\right) \\
\oplus \\
S\left(-6\right) 
\end{matrix}
\overset{[\Delta_1 \ -\Delta_2 \ \Delta_3]}{\longrightarrow}
S \overset{\varepsilon}{\longrightarrow} R \longrightarrow 0
\end{equation*}
gives a graded minimal free resolution of $R$, where $\Delta_1 = -YZ$, $\Delta_2 = XY$, and $\Delta_3 = Z^3-X^2$. Hence we get the resolution of $\rmK_R$ below
 \begin{equation*}
0 \longrightarrow
S\left(-6\right) \\
\overset{\ 
\left[\begin{smallmatrix}
\ \  \Delta_1 \\
-\Delta_2 \\
\ \ \Delta_3
\end{smallmatrix}\ \right] \ 
}{\longrightarrow}
\begin{matrix}
S\left(-3\right) \\
\oplus \\
S\left(-2\right) \\
\oplus \\
S 
\end{matrix}
\overset{\ 
\left[\begin{smallmatrix}
Z^2 & X & Y\\
X & Z & 0 \\
\end{smallmatrix}\right]
}{\longrightarrow}
\begin{matrix}
S\left(-1\right) \\
\oplus \\
S 
\end{matrix} \overset{\varepsilon}{\longrightarrow} \rmK_R \longrightarrow 0.
\end{equation*}
This shows $\rma(R) = 1$ and $[\rmK_R]_{-1} = k \xi$, where $\xi = \varepsilon \left(
\left(\begin{smallmatrix}
1 \\
0
\end{smallmatrix}\right)\right)$. Hence, for each homomorphism $\varphi : R \to \rmK_R(-1)$ of graded $S$-modules with $\varphi \ne 0$, we see that $\Im \varphi = R \xi$. Therefore
$$
\left(\rmK_R/R\xi \right)(-1) \cong S/(X, Z)
$$
which implies the map $\varphi$ is not injective; see \cite[Lemma 3.1 (2)]{GTT}. So $R$ is not almost Gorenstein as a graded ring. On the flip side, the elementary row operation
$$
\begin{pmatrix}
Z^2 & X & Y \\
X & Z & 0
\end{pmatrix} \  \ \longrightarrow \  \ 
\begin{pmatrix}
Z^2+X & X+Z & Y \\
X & Z & 0
\end{pmatrix}
$$
and the equality $(Z^2+X, X+Z, Y) = (X, Y, Z)$ in the local ring $S_N$ where $N=(X, Y, Z)S$ guarantee that $R_M$ is an almost Gorenstein local ring by \cite[Theorem 7.8]{GTT}.
\end{proof}

Example \ref{3.3} shows Theorem \ref{main} is no longer true even when $R$ is  reduced. As we show next, there is a counterexample of Theorem \ref{main} in case of homogenous reduced rings as well.

\begin{ex}
Let $S=k[X, Y, Z]$ be the polynomial ring over a field $k$. We consider $S$ as a $\bbZ$-graded ring under the grading $k=S_0$ and $X, Y, Z \in S_1$. Set $R=S/I$, where $I = (X, Y) \cap (Y, Z) \cap (Z, X) \cap (X, Y+Z)$. 
Then $R$ is not an almost Gorenstein graded ring, but the local ring $R_M$ is almost Gorenstein, where $M$ denotes the graded maximal ideal of $R$.
\end{ex}

\begin{proof}
Note that $I$ is an radical ideal of $R$ and $I=(XY, XZ, YZ(Y+Z)) = 
\rmI_2\left(\begin{smallmatrix}
Y+Z & 0 & Y \\
0 & X & YZ
\end{smallmatrix}\right)$. Then the homogeneous ring $R$ is Cohen-Macaulay,   reduced, and of dimension one. Set $\Delta_1 = XY$, $\Delta_2 = YZ(Y+Z)$, and $\Delta_3 = X(Y+Z)$. 
Since 
\begin{equation*}
0 \longrightarrow
\begin{matrix}
S\left(-3\right) \\
\oplus \\
S\left(-4\right) 
\end{matrix}
\overset{\ 
\left[\begin{smallmatrix}
Y+Z & 0 \\
0 & X \\
Y & YZ
\end{smallmatrix}\right] \ 
}{\longrightarrow}
\begin{matrix}
S\left(-2\right) \\
\oplus \\
S\left(-3\right) \\
\oplus \\
S\left(-2\right) 
\end{matrix}
\overset{[\Delta_1 \ -\Delta_2 \ \Delta_3]}{\longrightarrow}
S \overset{\varepsilon}{\longrightarrow} R \longrightarrow 0
\end{equation*}
forms a graded minimal free resolution of $R$, we get the resolution 
 \begin{equation*}
0 \longrightarrow
S\left(-3\right) \\
\overset{\ 
\left[\begin{smallmatrix}
\ \  \Delta_1 \\
-\Delta_2 \\
\ \ \Delta_3
\end{smallmatrix}\ \right] \ 
}{\longrightarrow}
\begin{matrix}
S\left(-1\right) \\
\oplus \\
S \\
\oplus \\
S \left(-1\right)
\end{matrix}
\overset{\ 
\left[\begin{smallmatrix}
Y+Z & 0 & Y\\
0 & X & YZ \\
\end{smallmatrix}\right]
}{\longrightarrow}
\begin{matrix}
S \\
\oplus \\
S \left(-1\right)
\end{matrix} \overset{\varepsilon}{\longrightarrow} \rmK_R \longrightarrow 0.
\end{equation*}
of $\rmK_R$ as a graded $S$-module. Thus $\rma(R)=1$ and $[\rmK_R]_{-1} = k \xi$, where $\xi = \varepsilon \left(
\left(\begin{smallmatrix}
0 \\
1
\end{smallmatrix}\right)\right)$. 
We have the elementary row operation
$$
\begin{pmatrix}
Y+Z & 0 & Y \\
0 & X & YZ
\end{pmatrix} \  \ \longrightarrow \  \ 
\begin{pmatrix}
Y+Z & X & Y+YZ \\
0 & X & YZ
\end{pmatrix}
$$
and the equality $(Y+Z, X, Y+YZ) = (X, Y, Z)$ in $S_N$ where $N=(X, Y, Z)S$. 
Similarly as in the proof of Example \ref{3.3}, we conclude that $R$ is not almost Gorenstein as a graded ring; while the local ring $R_M$ is almost Gorenstein. 
\end{proof}



\begin{thebibliography}{20}

\bibitem{BF}
{\sc V. Barucci and R. Fr\"{o}berg}, One-dimensional almost Gorenstein rings, {\em J. Algebra}, {\bf 188} (1997), no. 2, 418--442.

\bibitem{CELW}
{\sc E. Celikbas, N. Endo, J. Laxmi, and J. Weyman}, Almost Gorenstein determinantal rings of symmetric matrices, {\em Comm. Algebra}, {\bf 50} (2022), no.12, 5449--5458.

\bibitem{E}
{\sc N. Endo}, How many ideals whose quotient rings are Gorenstein exist$?$, arXiv:2305.19633.

\bibitem{Eto}
{\sc K. Eto}, Almost Gorenstein monomial curves in affine four space, {\em J. Algebra}, {\bf 488} (2017), 362--387. 

\bibitem{GKMT}
{\sc S. Goto, D. V. Kien, N. Matsuoka, and H. L. Truong}, Pseudo-Frobenius numbers versus defining ideals in numerical semigroup rings, {\em J. Algebra}, {\bf 508} (2018), 1--15.

\bibitem{GMP}
{\sc S. Goto, N. Matsuoka, and T. T. Phuong},  Almost Gorenstein rings, {\em J. Algebra}, {\bf 379} (2013), 355--381.


\bibitem{GMTY2}
{\sc S. Goto, N. Matsuoka, N. Taniguchi, and K.-i. Yoshida}, The almost Gorenstein Rees algebras over two-dimensional regular local rings, {\em J. Pure Appl. Algebra}, {\bf 220} (2016), 3425--3436.






\bibitem{GTT}
{\sc S. Goto, R. Takahashi, and N. Taniguchi}, Almost Gorenstein rings -towards a theory of higher dimension, {\em J. Pure Appl. Algebra}, {\bf 219} (2015), 2666--2712.


\bibitem{GW}
{\sc S. Goto and K.-i. Watanabe}, On graded rings I, {\em J. Math. Soc. Japan}, {\bf 30} (1978), no. 2, 179--213.





\bibitem{N}
{\sc H. Nari}, Symmetries on almost symmetric numerical semigroups, {\em Semigroup Forum}, {\bf 86} (2013), no.1, 140--154.


\bibitem{Sally}
{\sc J. Sally}, Cohen-Macaulay local rings of maximal embedding dimension, {\em J. Algebra}, {\bf 56} (1979), 168--183.

\bibitem{T}
{\sc N. Taniguchi}, On the almost Gorenstein property of determinantal rings, {\em Comm. Algebra}, {\bf 46} (2018), no.3, 1165--1178.


\bibitem{ZS}
{\sc O. Zariski and P. Samuel}, Commutative Algebra Volume II, {\em Springer}, 1960.


\end{thebibliography}
\end{document}